\documentclass[12pt]{article}

\usepackage{amsmath}
\usepackage{amsthm}
\usepackage{amssymb}
\usepackage{braket} 
\usepackage{mathtools} 
\usepackage{dsfont} 
\usepackage[all,cmtip,2cell]{xy}

\usepackage{color}

\newtheorem{thm}{Theorem}[section]
\newtheorem{theorem}[thm]{Theorem}
\newtheorem{corollary}[thm]{Corollary}
 
\newtheorem{proposition}[thm]{Proposition}

\theoremstyle{definition}
\newtheorem{definition}[thm]{Definition}
\newtheorem{example}[thm]{Example}
\newtheorem{observation}[thm]{Observation}
\newtheorem{remark}[thm]{Remark}

\newcommand{\Sets}{\mathrm{(Set)}}
\newcommand{\sets}{\mathrm{(Set)}}
\newcommand{\Hom}{\mathrm{Hom}}

\newcommand{\Ob}{\mathrm{Ob}}
\newcommand{\pr}{\mathrm{pr}}
\newcommand{\Cov}{\mathrm{Cov}}
\newcommand{\SMan}{\mathrm{(SMan)}}

\newcommand{\SSpaces}{\mathrm{(SSpaces)}}
\newcommand{\SSpacesk}{\mathrm{(SSpaces)}_k} 
\newcommand{\SSpacesLk}{\mathrm{(SSpaces^L)}_k} 
\newcommand{\SSpacesL}{\mathrm{(SSpaces^L)}} 
\newcommand{\SSlfg}{\SSpacesk^{\textnormal{lfg}}}
\newcommand{\Sch}{\mathrm{(Sch)}}
\newcommand{\SSch}{\mathrm{(SSch)}}

\newcommand{\szs}{superspace site}
\newcommand{\szss}{superspace sites}
\newcommand{\Y}{\mathrm{Yon}}
\newcommand{\op}{\mathrm{op}}
\newcommand{\id}{\mathrm{id}}
\newcommand{\cC}{\mathcal{C}}

\newcommand{\cO}{\mathcal{O}}
\newcommand{\cT}{\mathcal{T}}

\newcommand{\cX}{\mathcal{X}}

\newcommand{\R}{\mathbb{R}}
\newcommand{\C}{\mathbb{C}}
\newcommand{\N}{\mathbb{N}} 

\newcommand{\lra}{\longrightarrow}

\begin{document}

\centerline{\Large\bf Representability in Supergeometry}

\medskip
\centerline{R. Fioresi$^\sharp$, F. Zanchetta$^\star$}

\medskip
\centerline{\it $^\sharp$ Dipartimento di Matematica, Universit\`a di Bologna}
\centerline{\it Piazza di Porta San Donato 5, 40127 Bologna, Italy}
\centerline{\footnotesize e-mail: rita.fioresi@unibo.it}

\medskip
\centerline{\it $^\star$ Mathematics Institute,
Zeeman Building}
\centerline{\it University of Warwick,
Coventry CV4 7AL, England}
\centerline{\footnotesize 
 e-mail: F.Zanchetta@warwick.ac.uk}

\begin{abstract} 
In this paper we use the notion of Grothendieck
topology to present a unified way to approach 
representability in supergeometry, which applies to both
the differential and algebraic settings.
\end{abstract}

\newpage

\section{Introduction}
Supergeometry is the mathematical tool originally developed to study
supersymmetry. It was discovered, in the early 1970s, 
by the physicists Wess and Zumino (\cite{WZ}) and 
Salam and Strathdee (\cite{SS}) among others. 
Supergeometry grew out of the works 
of Berezin (\cite{Ber}), Kostant(\cite{Kon}) and Leites (\cite{leites}), then,
later on, by Manin (\cite{ma}), Bernstein (\cite{dm})  
and others. These authors introduced an algebraic point of view 
on differential geometry, with emphasis on the methods that were 
originally developed in algebraic geometry by Grothendieck 
to handle schemes. In particular, the functor of points approach 
invented by Grothendieck  turned out to be very useful to formalize 
the physical  ``anticommuting variables'' of supersymmetry and provided a 
very useful tool to link algebra and geometry in a categorical way. 
The language developed by Grothendieck is, in fact, powerful enough to 
reveal the geometric nature not just of superschemes, 
but also of supermanifolds and superspaces in general. 

\medskip
In this paper we want to examine representability in the
supergeometric context, with the use of Grothendieck topologies.
In particular, we are able to prove a representability criterion
(see Theorem \ref{rep-crit}), that can be applied to functors 
$\cC^\op \lra \Sets$, where $\cC$ is a  
{\sl superspace site}, that is a full subcategory of the category 
of superspaces $\SSpaces$, with some additional very natural properties
(see Def. \ref{szs-def}). This broadens the range of application
of the criterion, first published in
\cite{ccf}, to include, not only superschemes or supermanifolds, but
also some less trivial categories like Leites regular supermanifolds
and locally finitely generated superspaces, introduced by Alldridge
et al. in \cite{All}. Our hope is that more general objects can
be studied using this criterion, which formalizes the ideas of
Grothendieck, adapting them to the supergeometric context.
  
\medskip
{\bf Acknoledgements}. We wish to thank A. Alldridge and D. Leites
for useful discussions. We thank the anonymous referee for very valuable
comments. 

\section{Preliminaries on Grothendieck topologies} 
\label{prelim-sec}

We start with the notion of {\sl Grothendieck topology}. For more
details we refer the reader to 
\cite{SGA4}, Expos\'e ii,
\footnote{In  \cite{SGA4} what we call ``Grothendieck
topology'' is called ``pretopology'',
we adhere to the terminology in \cite{Vis}.}
\cite{Vis},  \cite{SP}.

\begin{definition}\label{gtop-def} 
Let $\mathcal{C}$ be a category. 
A \emph{Grothendieck topology} $\mathcal{T}$ on $\mathcal{C}$ 
assigns to each object $U\in \Ob(\mathcal{C})$ a collection 
$\Cov(U)$ whose elements are families of morphisms with fixed target $U$,
with the following properties.
\begin{itemize}
\item[(1)] If $V\rightarrow U$ is an isomorphism, 
$\lbrace V\rightarrow U\rbrace\in \Cov(U)$.
\item[(2)] If $V\rightarrow U$ is an arrow and $\lbrace U_i\rightarrow 
U\rbrace_{i\in I}\in \Cov(U)$, the fibered products 
$\lbrace U_i\times_U V\rbrace$ exist and the collection of 
projections $\lbrace U_i\times_U V\rightarrow V\rbrace_{i\in I}\in \Cov(V)$.
\item[(3)] If $\lbrace U_i\rightarrow U\rbrace_{i\in I}\in \Cov(U)$ and 
for each $i$ we have that $\lbrace V_{ij}\rightarrow U_i\rbrace_{j\in J_i}\in 
\Cov(U_i)$, then $\lbrace V_{ij}\rightarrow U_i\rightarrow 
U\rbrace_{i\in I,j\in J_i}\in \Cov(U)$.
\end{itemize}
The pair $(\mathcal{C},\mathcal{T})$ is called a \emph{site}. 
The elements of $\Cov(U)$ are called \emph{coverings}.
\end{definition}

We may abuse the notation and write $\mathcal{U}\in \mathcal{T}$ 
or $\mathcal{U}\in \Cov(\mathcal{C})$ to indicate that 
$\mathcal{U}=\lbrace U_i\rightarrow U\rbrace_{i\in I}$ is a covering 
for the topology $\mathcal{T}$.

\medskip
Now we discuss some key examples, which will be fundamental
for our treatment.

\begin{example} \label{gtop-ex}

\begin{enumerate} 

\item
Let us consider a topological space $X$ and set $\cX_{cl}$ 
to be the category with open sets as objects 
and inclusions as arrows. We say 
$\lbrace U_i\rightarrow U\rbrace_{i\in I}\in \Cov(U)$ if and only if 
$\bigcup_{i\in I}U_i=U$. We obtain a site $(X_{cl},\cX_{cl})$.

\item Let us consider the category $\Sch$ of schemes and define coverings 
of $U$ to be collections of open embeddings whose images cover $U$. 
This is a topology, because of the existence and the properties of the 
fibered product in $\Sch$. This is called the \textit{Zariski topology}.

\end{enumerate}

\end{example}

Now we want to compare different topologies on the same
category.
\begin{definition} \label{ref-def}
Let
$\mathcal{C}$ be a category and let 
$\mathcal{U}$ $=\lbrace U_i\rightarrow U\rbrace_{i\in I}$, 
$\mathcal{V}$ $=\lbrace V_j\rightarrow U\rbrace_{j\in J}$ be two families of 
arrows with fixed target. We say that
$\mathcal{V}$ is a \emph{refinement} of $\mathcal{U}$ 
if for every $j\in J$ there exists $i\in I$ such that 
$V_j\rightarrow U$ factors through $U_i\rightarrow U$. 
\end{definition}

We say that the topology $\mathcal{T}$ is \emph{subordinate} to 
the topology $\mathcal{T'}$  if every covering in $\mathcal{T}$ 
has a refinement that is a covering in $\mathcal{T'}$ and 
we write $\mathcal{T}\prec \mathcal{T'}$. If $\mathcal{T}\prec \mathcal{T'}$ 
and $\mathcal{T'}\prec \mathcal{T}$ we say that 
$\mathcal{T}$ and $\mathcal{T'}$ are \emph{equivalent} i. e. 
$\mathcal{T}\equiv \mathcal{T'}$.

\medskip
We are now in the position to define what a sheaf is in this framework.

Let  $(\mathcal{C},\mathcal{T})$ be a site, 
$\mathcal{U}=\lbrace U_i\rightarrow U\rbrace_{i\in I}$ a covering. 
Consider $F:\mathcal{C}^{\op}\rightarrow \mathrm{(Set)}$ and the fibered product 
$U_i\times_U U_j$.
We denote the projection morphisms as
$$
\pr_1:U_i\times_U U_j\rightarrow U_i\quad\mathrm{and}\quad 
\pr_2:U_i\times_U U_j\rightarrow U_j
$$ and the pullback morphisms $F(\pr_1)$, $F(\pr_2)$ as 
$\pr_1^{\ast}$ and $\pr_2^{\ast}$.

\begin{definition}
Let $(C,\mathcal{T})$ be a site and consider 
$F:\mathcal{C}^{\op}\rightarrow \mathrm{(Set)}$. 
\begin{itemize}
\item[(1)] $F$ is \emph{separated} if, given $\mathcal{U}
=\lbrace U_i\rightarrow U\rbrace_{i\in I}
\in \mathcal{T}$ and 
$x,y\in F(U)$ so that their pullback to $F(U_i)$ coincide for every $i$, 
we have $x=y$.
\item[(2)] $F$ is a \emph{sheaf} if, for every covering 
$\mathcal{U}=\lbrace U_i\rightarrow U\rbrace_{i\in I}\in \mathcal{T}$, the diagram
\begin{displaymath}
\xymatrix{
F(U)\ar[r] & \prod_{i\in I}F(U_i) \ar@<1ex>[r] 
\ar@<-1ex>[r]
& \prod_{(i,j)\in I\times I}F(U_i\times_U U_j) 
}
\end{displaymath}
is the diagram of an equalizer, in other words the first arrow is injective 
with image 
$$
\lbrace(\xi_i)\in \prod_{i}F(U_i)\:|\:\pr_1^{\ast}(\xi_i)=
\pr_2^{\ast}(\xi_i)\: \mathrm{ on }\:\prod_{i,j}F(U_i\times_U U_j)\rbrace
$$
\end{itemize}
where $\pr^*_i=F(\pr_i)$. 
If 
$\mathcal{U}=\set{f_i:U_i\rightarrow U}
\in \mathcal{T}$, we call the morphisms $F(f_i):F(U)\rightarrow F(U_i)$
\emph{restrictions}.  Notice that
the morphism $F(U)$ $\rightarrow$  $\prod_{i\in I}F(U_i)$ 
is induced by restrictions. 
\end{definition}

We have the following proposition (see \cite{Vis}).

\begin{proposition}
Let $\mathcal{C}$ be a category  and $\mathcal{T}$, $\mathcal{T'}$ 
two topologies on $\cC$. If   $\mathcal{T}\prec\mathcal{T'}$, 
then a sheaf in $\mathcal{T'}$ is also a sheaf in $\mathcal{T}$. 
Furthermore if $\mathcal{T}\equiv\mathcal{T'}$, then the two topologies 
have the same sheaves.
\end{proposition}

We now introduce the {\sl slice category}, which will be instrumental
to discuss superspaces over a base superspace.

\begin{definition}\label{slice-def}
Let $\mathcal{C}$ be a category and let 
$S\in\Ob(\mathcal{C})$. We call $\mathcal{C}/S$ 
the \emph{slice category of $\mathcal{C}$ 
over $S$}. This is  the category with objects the pairs
consisting of an object $X$ in $\cC$ and an arrow $f:X\rightarrow S$.
Given two objects $f:X\rightarrow S$ and $h:Y\rightarrow S$, 
a morphism between them is an arrow 
$g:X\rightarrow Y$ of $\mathcal{C}$ such that $h\circ g=f$.
\end{definition}

We now want to lift a Grothendieck topology from a site 
$(\mathcal{C},\mathcal{T})$ to a slice category $\mathcal{C}/S$. 

\begin{definition} \label{slice-top}
Let $(\mathcal{C},\mathcal{T})$ be a site and let $\mathcal{C}/S$ 
be the slice category of $\mathcal{C}$ over $S$. The \emph{slice topology} 
$\mathcal{T}/S$ on $\mathcal{C}/S$ is the topology which has as coverings of 
an object $X/S$ families of morphisms $\mathcal{U}/S=
\set{\varphi_i:U_i\rightarrow X}$, where the $\varphi_i$ are
morphisms in  $\mathcal{C}/S$, and $\mathcal{U}/S\in\mathcal{T}$.
\end{definition}
One can readily check that $\mathcal{T}/S$ 
is a topology over $\mathcal{C}/S$. 

\begin{definition} \label{can-top}
Let us consider a category $\mathcal{C}$ and let $\mathcal{T}$ 
be a Grothendieck topology on it. $\mathcal{T}$ is said to 
be \emph{subcanonical}, 
if all the representable functors are sheaves with respect 
to the topology $\mathcal{T}$. Recall that $F:\mathcal{C}^{\op} \lra \sets$
is \textit{representable} if there exists an object $X \in 
\mathrm{Ob}(\mathcal{C})$ 
such that $F\cong h_X$, where $h_X(Y)=\Hom(Y,X)$ and $h_X(f)\phi=f \circ \phi$.
 $\mathcal{T}$ is called \emph{canonical} if is subcanonical
and every subcanonical topology is subordinate to it. 
\end{definition}

One can show that if $\mathcal{T}$ is subcanonical on $\cC$, 
then $\mathcal{T}/S$ is subcanonical on $\cC/S$
(see \cite{Vis}, Proposition 2.59).

\section{Superspaces and functor of points}\label{sspaces-sec}

We introduce the category of superspaces and some important
subcategories. For all of the terminology and notation
we refer the reader to \cite{dm}, \cite{ma}, \cite{vsv}, \cite{ccf}.

\begin{definition}
A \emph{superspace} $X$ is a topological space $|X|$ 
endowed with a sheaf (in the usual sense) of commutative super rings 
$\mathcal{O}_X$, which is called the \emph{structure sheaf} of $X$,  so that 
for every $p\in |X|$, the stalk $\mathcal{O}_{X,p}$ is a local super ring. 
\end{definition}
Superspaces forms a category denoted by $\SSpaces$ whose morphisms are 
given as follows.
Given two superspaces $(|X|,\mathcal{O}_X)$ and $(|Y|,\mathcal{O}_Y)$, 
an arrow between them is a pair $f=(|f|,f^{\ast})$, so that:
\begin{itemize}
\item[(1)] $|f|:|X|\rightarrow |Y|$ is a continuous map.
\item[(2)] $f^{\ast}:\mathcal{O}_Y\rightarrow f_{\ast}\mathcal{O}_X$ is a map 
of sheaves of super rings. 
\item[(3)] The map of local super rings 
$f^{\ast}_p:\mathcal{O}_{Y,|f|(p)}\rightarrow \mathcal{O}_{X,p}$ 
is a local morphism for all $p$.
\end{itemize}

If $X$ is a superspace, we can define its \emph{functor of points}:
$$
h_X: \SSpaces^{\op} \lra \Sets, \quad h_X(T)=\Hom(T,X), \quad
h_X(\phi)(f)=f\circ \phi.
$$
One of the main purposes of this note is to establish when a functor 
$F: \cC^{\op} \lra \sets$, where $\cC$ is a
full subcategory of $\SSpaces,$ is the functor of points of a superspace
or more precisely of an object in the category $\cC$.

\medskip
Following Def. \ref{slice-def}, 
we can define the category of
\emph{superspaces over a base superspace $S$}. The objects
of this category are pairs consisting of 
a superspace $X$ together with a morphism $X \lra S$. Morphisms are defined
accordingly, as the morphisms of superspaces that commute with the given
morphisms to $S$. A special case of particular interest to us is when 
$S=(|S|, k)$, where $|S|$ is a point and the structural
sheaf is just a field $k$. We call
the superspaces over such an $S$, \textit{$k$-superspaces} 
and we denote them as $\SSpacesk$. 
We shall see more on this later.

\begin{remark}
We observe that $k$-superspaces are simply superspaces whose 
structure sheaf is a sheaf of $k$-superalgebras. Moreover the morphisms 
of $\SSpacesk$ are the superspace morphisms which preserve 
the $k$-superalgebra structure, i.e. which are $k$-linear. 
\end{remark}

Consider a superspace $X$. If $|U|$ is open in $|X|$,
we can define the superspace $X_U=(|U|,\mathcal{O}_{X|U})$, together with
the canonical morphism $j_{X|U}:X_U\rightarrow X$. We say that a morphism
of superspaces $\varphi:Y \rightarrow X$ is an \emph{open embedding}, 
if there exist an isomorphism
$\phi :Y\rightarrow X_{U}$ and 
$\varphi =j_{X|U}\circ \phi$. 
We denote by $\varphi (Y)$ the superspace $X_{U}$.
It makes sense to speak about union and intersections of open 
subsuperspaces $U$ and $V$ and
we shall write, with an abuse of notation, 
$U \cup V$ and $U \cap V$ to denote them.

A morphism of superspaces $f:X\rightarrow Y$ is said to be 
a \emph{closed embedding} if $|f|$ is a closed embedding 
and the morphism $f^{\ast}:\mathcal{O}_Y\rightarrow f_{\ast}\mathcal{O}_X$ 
is surjective, 
i.e. it induces an isomorphism 
$$
\cO_Y/\mathcal{I}_X\rightarrow f_{\ast}\cO_X, \qquad \mathcal{I}_X:=ker(f^{\ast}).
$$ 
We call $\mathcal{I}_X$ the \emph{vanishing ideal} of $f$ or of $X$.

We say that a morphism $f:X\rightarrow Y$ between 
superspaces is an \emph{embedding} if $f=g\circ h$ where $h$ is 
a closed embedding and $g$ is an open embedding. 
See \cite{All} for more details. 

\medskip
Open embeddings are particularly important because of the following fact, 
which is stated in the case of superspaces in \cite{All} and 
for ringed spaces in \cite{EGA}, Chapitre 0, 4.5.2.

\begin{proposition}\label{fiber-prop}
Let $i : U\rightarrow X$ be an open embedding of superspaces. 
Then for every morphism $f: Y\rightarrow X$ of superspaces, 
the fibered product $U \times_X Y$ exists and: 
$$
U\times_XY= 
(|f|^{-1}(|X_{U}|), 
\mathcal{O}_{Y||f|^{-1}(|X_{U}|)} ) 
=Y_{|f|^{-1}(U)}
$$
Furthermore, $\pr_2:U\times_X Y \rightarrow Y$ is an open embedding.
\end{proposition}

We now define a Grothendieck topology on $\SSpaces$.

\begin{definition}
Consider the category $\SSpaces$ and for $X \in \Ob(\SSpaces)$ 
define $\Cov(X)$ as 
the collections 
$$
\mathcal{U}=\lbrace \varphi_i:U_i\rightarrow X\rbrace_{i\in I}
$$ 
where $U_i$ are superspaces, 
the arrows $\varphi_i$ 
are open embeddings and 
$\bigcup_{i\in I}|U_i|=|X|$,
where, from now on, with an abuse of notation we write
$|U_i|$ in place of $|\varphi_i|(|U_i|)$.

As one can readily check, we obtain
a Grothendieck topology $\mathcal{S}$, that we call 
the \emph{global super topology}.
So we have a site $(\SSpaces,\mathcal{S})$.
\end{definition}

Superspaces can be built by local data, that are suitably
patched together.

\begin{definition}
A \emph{gluing datum} $((U_i)_{i\in I},( U_{ij})_{i,j\in I},(\phi_{ij})_{i,j\in I})$
of superspaces consists of:
\begin{itemize}
\item[$\bullet$] a collection of superspaces $ (U_i)_{i\in I}=(|U_i|, 
\cO_{U_i})_{i\in I}$; 
\item[$\bullet$] a collection of open super subspaces 
$( U_{ij}\subseteq U_i)_{i,j\in I}$ so that for each $i\in I$, $U_{ii}=U_i$.
\item[$\bullet$] a collection of isomorphisms 
$( \phi_{ji}:U_{ij}\rightarrow U_{ji})_{i,j\in I}$ so that the cocycle 
condition holds:
$$\phi_{ki}=\phi_{kj}\circ \phi_{ji}\quad \mathrm{on}\quad U_{ij}\cap U_{ik}$$
\end{itemize}
\end{definition}

If $((U_i)_{i\in I},( U_{ij})_{i,j\in I},(\phi_{ij})_{i,j\in I})$ is a gluing
datum, there exists a superspace $X=(|X|,\mathcal{O}_X)$ and a family of 
open embeddings $\lbrace \phi_i:U_i\rightarrow X\rbrace_{i\in I}$ so that
 $X=\bigcup_{i\in I}\phi_i(U_i)$,
$\phi_i=\phi_j\circ\phi_{ji}$ on $U_{ij}$ for all $i,j\in I$ and finally
$\phi_i(U_{ij})=\phi_j(U_{ji})=\phi_i(U_{i})\cap \phi_j(U_{j})$  for all $i,j\in I$.
$X$ and $\lbrace \phi_i:U_i\rightarrow X\rbrace_{i\in I}$ are uniquely 
determined up to isomorphism.

\medskip
There are some full subcategories of  $\SSpaces$ that are particularly
interesting. We start by examining the differential setting.

\begin{definition} \label{superm-def}
A \emph{real (resp. complex) superdomain} 
of dimension $(p|q)$ is a super ringed space 
$U^{p|q}=(U,\mathcal{O}_{U^{p|q}})$ where $U$ is an open subset 
of $\mathbb{R}^p$ (resp. $\C^p$) and 
$$
\mathcal{O}_{U^{p|q}}(V)= \cO_{U} 
\otimes \bigwedge (\theta^1, \dots,\theta^q)
$$ for $V$ open in $U$.  
$\cO_{U}= C^{\infty}_U$ is 
the sheaf of $\mathit{C}^{\infty}$ functions on $U$
(resp. $\cO_{U}=\mathcal{H}^{\infty}_U$ is the sheaf of holomorphic
functions on $U$), 
$\bigwedge (\theta^1, \dots, \theta^q)$ is the exterior algebra 
in $q$ odd coordinates $\theta^1, \dots, \theta^q$.

\medskip
Let be $M=(|M|,\mathcal{O}_M)$ a $k$-superspace
($k=\R$ or $\C$). We say that $M$ is a \emph{supermanifold} if
$M$ is locally isomorphic, as $k$-superspace, to a real
or complex superdomain.\footnote{In \cite{All} these are called
\emph{presupermanifolds}. In many references the definition of
supermanifold  requires more hypotheses (e.g. Hausdorff) on the
topological space.} 
Real
supermanifolds are called \emph{smooth supermanifolds}, while
complex supermanifolds are called \emph{holomorphic supermanifolds};
their categories are denoted as $\SMan_\R$ and $\SMan_\C$
or simply as $\SMan$, when we do not want to mark the difference
between the smooth and holomorphic treatment.
\end{definition} 

\medskip
As for superspaces we can easily define superschemes (or supermanifolds)
on a base superscheme (or supermanifold) $S$. We leave to the reader
the details.

\medskip
Two further interesting examples of superspaces are the 
{\sl Leites superspaces} and the {\sl locally finitely generated superspaces} 
(see \cite{All}). 

\begin{definition}
Let $X$ be an object of $\SSpaces_k$ ($k=\R$ or $\C$). 
We say that $X$ is a \emph{Leites regular superspace} if for every open 
subsuperspace $U\subseteq X$, and every integer $p$, the following map 
is a bijection
$$
\Psi :\Hom (U,k^p)\rightarrow (\cO_X(U)_{0})^p,\quad 
\varphi\mapsto (\varphi^*(t_1), \dots,\varphi^*(t_p)) 
$$
Here $k^p$ is a superdomain, $t_1,...,t_p$ are the canonical coordinates 
of $k^p$.
We denote their category with $\SSpacesLk$.
\end{definition}

Leites superspaces are a full subcategory of $\SSpacesk$ 
and open subsuperspaces of Leites regular superspaces are still 
Leites regular. Moreover, we have immediately that 
supermanifolds are Leites regular superspaces, because of Chart 
Theorem (see \cite{ma} Ch. IV).

\begin{remark}
Leites superspaces essentially describe the superspaces for which the 
Chart Theorem holds. Furthermore, it is possible to give a definition of 
Leites superspaces which comprehends not only differentiable or 
holomorphic supermanifolds, but also the real analytic ones. 
We do not introduce them here, referring the reader to \cite{All} 
for more details; their treatment is similar to the other ones.
\end{remark}

We can now define the category of locally finitely generated $k$-superspaces. 
First of all, we need the definition of tidy embedding.

\begin{definition}
Consider an embedding $\varphi:Y\rightarrow X$ of superspaces 
with vanishing ideal $\mathcal{I}$. We say that $\varphi$ is \emph{tidy} 
if for every $x\in \varphi(|Y|)$, every open 
neighbourhood $U\subseteq |X|$ of $x$ and every $f\in\cO_X(U)$, we have:
$$
( \forall y\in U\cap\varphi (|Y|),r\in\N\, :\, f_y\in\mathcal{I}_y+\mathfrak{m}^r_{X,y}) \, \Rightarrow \, f_x\in\mathcal{I}_x
$$
($f_x$ is the image of $f$ in $\cO_{X,x}$). 
We say that a superspace $X$ is \textit{tidy} if id$_X$ is a tidy embedding.
\end{definition}

For more about tidiness, we refer the reader to \cite{All}, Section 3.4. 
We note that open subsuperspaces of tidy superspaces are tidy.

\medskip
We can now recall the definition of locally finitely generated superspace.

\begin{definition}
Let $X$ be an object of $\SSpacesk$. We say that $X$ is 
\emph{finitely generated} if there exists a tidy embedding 
$\varphi: X\rightarrow k^{p|q}$. We say that $X$ is 
\emph{locally finitely generated} if it admits a cover by open subsuperspaces 
which are finitely generated. We denote the full subcategory 
of $\SSpacesk$ of locally finitely generated superspaces as $\SSlfg$.
\end{definition}

Note that open subsuperspaces of locally finitely generated superspaces 
are still in $\SSlfg$. The following
proposition is one of the main results in 
\cite{All} (see Sec. 5). 

\begin{proposition}\label{subcat-prop}
$\SSlfg$ is a full subcategory of the category of Leites regular 
superspaces; it is finitely complete 
(i.e. all the finite limits exist in it) 
and $\SMan_k$ is a full subcategory of $\SSlfg$. Moreover finite limits
are preserved by the inclusion.
\end{proposition}

\begin{observation} \label{subcatalg-obs}
Notice that the following categories are all full subcategories of $\SSpaces_k$:
$\SMan_k$, $\SSpacesLk$, $\SSlfg$.
Similarly, the categories $\SMan_k/S$, $\SSpacesLk/S$, $\SSlfg/S$
are full subcategories of $\SSpaces_k/S$. 
\end{observation}

We now turn to examine the algebraic category.

\begin{definition}
A \emph{superscheme} $X$ is a superspace $(|X|,\mathcal{O}_X)$ 
where $(|X|, \mathcal{O}_{X,0})$ is an ordinary scheme and $\mathcal{O}_{X,1}$ 
is a quasi coherent sheaf of $\mathcal{O}_{X,0}$-modules.
Morphisms of superschemes are the morphisms of the corresponding
superspaces. We denote the category of superschemes by $\SSch$.
\end{definition}

Notice that $\SSch/S$ is a full subcategory of $\SSpaces/S$.

\section{Representability}\label{rep-sec}

In this section we examine a unified way to write representability
theorems in supergeometry. Given
a full subcategory $\cC$ of $\SSpaces$, the question we want to study is how to
characterize among all functors $F: \cC^{\op} \lra \sets$ those which
are the functor of points of an object in $\cC$.

\medskip
We start with the notion of \szs, which will be instrumental
for the general result on representability we want to give.

\begin{definition}\label{szs-def}
We say that a site $(\mathcal{C},\mathcal{T})$ is a \emph{\szs} if:

\begin{itemize}
\item[$\bullet$] $\mathcal{C}$ is a full subcategory of $\SSpaces$.
\item[$\bullet$] $\mathcal{T}$ is a topology
such that, for $\mathcal{U}=\lbrace f_i:U_i\rightarrow X\rbrace_{i\in I}
\in \cT$, we have that the arrows $f_i$ are open embeddings
and $\bigcup_{i\in I}|U_i|=|X|$. 
\item[$\bullet$] Given $X\in \Ob(\mathcal{C})$ and an open subset 
$|U|\subseteq |X|$, $X_{U}\in \Ob(\mathcal{C})$.
\item[$\bullet$] Given a gluing datum 
$((U_i)_{i\in I},( U_{ij})_{i,j\in I},(\phi_{ij})_{i,j\in I})$ in $\mathcal{C}$, 
and $X$ the corresponding superspace, 
with $\lbrace \phi_i:U_i\rightarrow X\rbrace_{i\in I}$ family of open embeddings, 
we have $X\in\Ob(\mathcal{C})$.
\end{itemize}
\end{definition}

Notice that the last two conditions say that 
\szs\:\:$\mathcal{C}$ is closed under the operation of gluing and 
under the operation of restriction to open subsets.
It is immediate to extend the previous definition 
to $\SSpaces/S$.

\begin{observation} \label{sites-obs}
Proposition \ref{subcat-prop} and Observation
\ref{subcatalg-obs}, together with Prop. 
\ref{fiber-prop}, 
give us that $\SMan_k$, $\SSpacesLk$, $\SSlfg$, $\SSch$  are
all examples of \szss,  where the topology is given by the open embeddings
i.e. it is the global super topology.
\end{observation}

We have the following proposition.

\begin{proposition}
Let be $(\mathcal{C},\mathcal{T})$ a \szs. Then $\mathcal{T}$ is subcanonical.
Hence every representable functor 
$F:\mathcal{C}^{\op}\rightarrow \mathrm{(Set)}$ is a sheaf for $\mathcal{T}$.
\end{proposition}

Now we would like to establish the converse, that is, we want to understand, 
when given a \szs\:\: $(\mathcal{C},\mathcal{T})$, a sheaf is representable,
in other words it is the functor of points of an object in $\cC$. 
This is a most relevant question in supergeometry, since very often
the only way to get hold of a supergeometric object is through
its functor of points. 

\begin{definition}
Given $F, G:\mathcal{C}^{\op}\rightarrow \mathrm{(Set)}$, 
a natural transformation $f: F \rightarrow G$ between them is called 
\emph{representable} or a \emph{representable morphism} 
if for every object $X$ in $\mathcal{C}$ and every natural transformation 
$g:h_X\rightarrow G$, the functor $F\times_Gh_X$ is representable.
\end{definition}

We now give the notion of open subfunctor. 

\begin{definition}\label{open-fun}
Let us consider the
functors $U, G:\mathcal{C}^{\op}\rightarrow \mathrm{(Set)}$ and a natural
transformation $f:U\rightarrow G$. We say that $U$ 
is an \emph{open subfunctor} of $G$, 
if $f$ is a monomorphism, it is representable
and for every natural transformation $g:h_X\rightarrow G$, $X \in \cC$,
the second projection $\pr_2:U\times_Gh_X\rightarrow h_X$, corresponds
to an open embedding $\pr_2^{\Y}:Z \lra X$, with $h_Z\cong U \times_G h_X$,
($\pr_2^{\Y}:Z \lra X$ corresponds to
the second projection of the fibered product
via the Yoneda's Lemma).
\end{definition}

\begin{definition}
Let us consider a \szs\:\: $(\mathcal{C},\mathcal{T})$ and take a 
functor $F:\mathcal{C}^{\op}\rightarrow \mathrm{(Set)}$. 
We say that a family $\lbrace f_i:U_i\rightarrow F\rbrace_{i \in I}$ 
of open subfunctors is 
an \emph{open covering} of $F$, if for every $X\in \Ob(\mathcal{C})$ 
and every natural transformation $g:h_X\rightarrow F$, 
the family $\lbrace (\pr_2)_i^\Y: X_i\rightarrow X\rbrace_{i\in I}$  
is a covering of $X$,
where the $X_i$ are the superspaces such that
 $h_{X_i}\simeq U_i\times_Fh_X$. 
If the $U_i$ are representable functors, we will say that $F$ has an 
\emph{open covering by representable functors}.
\end{definition}

\begin{observation}
Observe that, in a category $\mathcal{C}$, given $f:h_X\rightarrow h_Z$ and 
$g:h_Y\rightarrow h_Z$, then the fibered product $h_X\times_{h_Z}h_Y$ 
exists and it is representable if and only if $X\times_ZY$ 
exists in $\mathcal{C}$ and in this case, we have 
$h_X\times_{h_Z}h_Y\simeq h_{X\times_ZY}$. 
Then a morphism $f:h_X\rightarrow h_Z$ is representable if and only if, 
for every $g:h_Y\rightarrow h_Z$, 
$X\times_ZY$ exists. 
So we have that in every \szs, a morphism $f:h_X\rightarrow h_Z$ such that 
$f^\Y:X\rightarrow Z$ is an open embedding is always representable
($f^\Y$ denotes the morphism between superspaces corresponding to $f$
via the Yoneda's Lemma). 

\medskip
On the other hand
in the 
categories $\, \SSch$ and $\SSlfg$, fibered products always exist,
hence the morphism $f:h_X\rightarrow h_Z$ is always 
representable, even in the case in which it is not an open embedding.
We shall not need this fact in the
sequel, see  \cite{ccf} and \cite{All} for more details.
\end{observation}

We are now ready to state our representability criterion, 
which is essentially a rewriting of the work by
Grothendieck, in our context. Our proof 
generalizes the classical treatment in \cite{GW} (Theorem 8.9) and furthermore
provides a unified perspective on the corresponding results in \cite{ccf} 
(Theorems 10.3.7 and 9.4.3, see also
the footnote after Def. \ref{superm-def}), besides
including the categories $\SSpacesL$, $\SSlfg$ (see Prop. 
\ref{subcat-prop})
whose representability issues, to our knowledge, are not dealt with
elsewhere. The reader is also invited to read the 
statement and the proof of Proposition 4.5.4, 
Chapitre 0 of \cite{EGA}, making a comparison.

\begin{theorem}\label{rep-crit}
Let us consider a \szs\:\: $(\mathcal{C},\mathcal{T})$ and a functor 
$F:\mathcal{C}^{\op}\rightarrow \mathrm{(Set)}$. 
Then $F$ is representable if and only if both of the following hold:
\begin{itemize}
\item[(1)]$F$ is a sheaf for $( \mathcal{C},\mathcal{T})$.
\item[(2)]$F$ has a open covering by representable 
functors $\lbrace f_i:U_i\rightarrow F\rbrace_{i \in I}$.
\end{itemize}
\end{theorem}

\begin{proof}
Let us suppose that $F\simeq h_X$, i.e. $F$ is representable. 
Then by Proposition 3.1.5, $F$ is a sheaf and (1) holds. 
Moreover, we observe that the identity natural transformation 
$\id_F:F\rightarrow F$ is an open covering by representable functors and 
so also (2) holds.\\
Conversely, suppose that 
both (1) and (2) hold.
Denote the open covering by representable functors of $F$ by 
$\lbrace f_i:U_i\simeq h_{X_i}\rightarrow F\rbrace_{i \in I}$. 
We can assume without loss of 
generality that $\lbrace f_i:h_{X_i}\rightarrow F\rbrace_{i \in I}$ 
is our covering, i.e. that $U_i=h_{X_i}$.
We want to build an $X\in \Ob(\mathcal{C})$ such that 
$h_X\simeq F$. We do this by constructing an appropriate gluing datum 
$((V_i)_{i\in I},( V_{ij})_{i,j\in I},(\phi_{ij})_{i,j\in I})$. We set $V_i=X_i$. 
By the definition of open subfunctor, we have that $h_{X_i}\times_Fh_{X_j}$ 
is representable, and accordingly there exists an element 
$X_{ij}\in\Ob(\mathcal{C})$ such that 
$$
h_{X_i}\times_Fh_{X_j}\simeq h_{X_{ij}}
$$ 
with $\varphi_{ji}:=(\pr_2)_i^\Y:X_{ij}\rightarrow X_j$ an open 
embedding of superspaces. We set $V_{ji}=(\pr_2)_i^\Y(X_{ij})$.
The morphisms $f_i$ are open subfunctors, so they are injective
hence $(f_i)_Q:h_{X_i}(Q)\rightarrow F(Q)$ are also injective. 
So we can identify $(h_{X_i}\times_Fh_{X_j})(Q)$ and 
$(h_{X_j}\times_Fh_{X_i})(Q)$ with $h_{X_i}(Q)\cap h_{X_j}(Q)\subseteq F(Q)$ 
for all $Q$. Then $h_{X_j}\times_Fh_{X_i}=h_{X_i}\times_Fh_{X_j}$ 
and  we have $X_{ij}\simeq X_{ji}$. We define 
$\phi_{ij}=\varphi_{ij}\circ \varphi_{ji}^{-1}:V_{ji}\simeq V_{ij}$
To show that $((V_i)_{i\in I},( V_{ij})_{i,j\in I},( (\phi_{ij})_{i,j\in I})$ 
is a gluing datum, we have to check the cocycle condition.
Let  $X_{ijk}$ be the superspace
whose functor of points is 
$h_{X_{ijk}}:=h_{X_{i}}\times_F h_{X_{j}}\times_F h_{X_{k}}$.
The $X_{ijk}$'s are all isomorphic for any permutation
of the indices $i$, $j$, $k$, 
via the appropriate restrictions of $\phi_{ij}$. We leave the easy
details to the reader.
Consequently it is almost immediate to verify the cocycle condition 
$$
\phi_{ki}=\phi_{kj}\circ \phi_{ji}\quad \mathrm{on}\quad V_{ij}\cap V_{ik}.
$$
Since $((V_i)_{i\in I},( V_{ij})_{i,j\in I},( (\phi_{ij})_{i,j\in I})$ is a gluing 
datum we get a superspace $X$, which, by definition of \szs, 
is an object of $\mathcal{C}$. Notice that we have a covering 
$\mathcal{U}=\lbrace X_i\rightarrow X\rbrace_{i\in I}\in \mathcal{T}$, 
hence $\lbrace h_{X_i}\rightarrow h_{X}\rbrace_{i\in I}$ 
is an open covering by representable functors. 
We are left to prove that $F\simeq h_X$.
We construct a natural transformation $\eta : F\rightarrow h_X$. 
For each superspace $T\in \Ob(\mathcal{C})$, we need to give
a morphism $\eta_T :F(T)\rightarrow h_X(T)$.  
By Yoneda's Lemma, $F(T)\simeq \mathrm{Hom}(h_T,F)$, so we take 
$g\in \mathrm{Hom}(h_T,F)$. Consider the diagram:
\begin{displaymath}
\xymatrix{h_{Y_i}\simeq h_{X_i}\times_F h_T\ar[r] \ar[d]_{(pr_1)_i} & h_T \ar[d]^g\\
h_{X_i}\ar[r]_{f_i} & F
}
\end{displaymath}
We have $(Y_i \rightarrow T)_{i\in I}\in \Cov(T)$. By Yoneda's lemma, 
we obtain a family of morphisms $(\pr_1)_i^\Y:Y_i\rightarrow X_i$, 
that glue together, to give a morphism $t:T\rightarrow X$. 
So we set $\eta_T(g)=t$. 
\\
Now we have to build a natural transformation $\delta: h_X\rightarrow F$, 
which is the inverse of $\eta$. We define for each each superspace 
$T\in \Ob(\mathcal{C})$, a morphism $\delta_T :h_X(T)\rightarrow F(T)$.
Take $t\in h_X(T)$ and set $Y_i=t^{-1}(X_i)=X_i\times_XT$. 
This is the preimage of the open superspace $X_i$ in $X$ via the morphism $t$.
Then $(l_i:Y_i\rightarrow T)_{i\in I}\in \Cov(T)$. 
We obtain morphisms $g_i:Y_i\rightarrow X_i$, which correspond via 
Yoneda's Lemma to natural transformations $g_i':h_{Y_i}\rightarrow h_{X_i}=U_i$. 
The morphisms $f_i\circ g_i'$ glue together, because $F$ is a sheaf, 
to a morphism $g':h_T\rightarrow F$, which corresponds by Yoneda's Lemma 
to an element $g\in F(T)$. We define $\delta_T(t)=g$.
One can readily check that $\eta$ and $\delta$ are natural 
transformations one inverse of the other.
\end{proof}

Our result has a straightforward generalization to $\cC/S$,
we leave the details to the reader. 

\medskip
Now we can state a corollary of the previous theorem, 
that can be useful in applications.

\begin{corollary} \label{rep-cor}
Let be $(\mathcal{C},\mathcal{T})$ a \szs\:\: and let be $\mathcal{T'}$ a topology on $\mathcal{C}$ so that $\mathcal{T}\prec\mathcal{T'}$. 
If a functor $F:\mathcal{C}^{\op}\rightarrow \mathrm{(Set)}$ satisfies 
\begin{itemize}
\item[(1)]$F$ is a sheaf for $( \mathcal{C},\mathcal{T'})$,
\item[(2)]$F$ has an open covering by representable functors $\lbrace f_i:U_i\rightarrow F\rbrace_{i \in I}$,
\end{itemize}
then it is representable. Moreover this condition is also necessary 
for the representability of $F$ if $\mathcal{T'}$ is subcanonical.
\end{corollary}
\begin{proof}
If (1) and (2) hold, the hypothesis $\mathcal{T}\prec\mathcal{T'}$ 
implies that $F$ is a sheaf for $( \mathcal{C},\mathcal{T})$ and 
then $F$ is representable by the previous theorem. 
Moreover, if $\mathcal{T'}$ is subcanonical, then a representable 
functor is by definition a sheaf for the topology $\mathcal{T'}$, 
and $\id_F:F\rightarrow F$ is a open covering by representable functors, 
then (1) and (2) are satisfied.
\end{proof}

The following observation is important for the
applications of our main result \ref{rep-crit}.

\begin{observation} 
The representability criterion \ref{rep-crit} 
holds for the sites detailed in Obs. \ref{sites-obs}.
In other words, to prove that a functor
$F: \cC^\op \lra \Sets$ is the functor of points of an object
in $\cC$, where $(\cC, \cT)$ is a superspace site, it is enough
to verify the properties (1) and (2) of \ref{rep-crit}. 
\end{observation}


\begin{thebibliography}{99}

\bibitem[All13]{All} A. Alldridge, J. Hilgert, T. Wurzbacher, 
\textit{Singular superspaces}. \emph{ Math. Z.},  278 (2014),
441-492.


\bibitem[SGA4]{SGA4} M. Artin, A. Grothendieck, J. L. Verdier, 
\emph{Th\'eorie des topos et cohomologie \'etale des schemas 1, 2, 3}, 
Springer-Verlag, 1972.


\bibitem[BCF10]{bcf1} L. Balduzzi, C. Carmeli, R. Fioresi, 
\textit{The local Functor of points of supermanifolds}. 
\emph{Exp. Math.}, Vol. 28 (3): 201-217, 2010.

\bibitem[BCF12]{bcf2} L. Balduzzi, C. Carmeli, R. Fioresi, 
\textit{A comparison of the functors of points of supermanifolds}. 
\emph{Journal of Algebra and its Applications}, Vol. 12 (3), 2012.


\bibitem[Be87]{Ber}
F. A. Berezin, \emph{Introduction to superanalysis}. 
Math. Phys. Appl. Math. 9, D. Reidel Publishing C., Dordrecht, 1987.


\bibitem[CCF11]{ccf} C. Carmeli, L. Caston, R. Fioresi, 
\emph{Mathematical Foundations of Supersymmetry}, 
with an appendix with I. Dimitrov. EMS Ser. Lect. Math.,
European Math. Soc., Zurich, 2011.

\bibitem[StPr]{SP} J. De Jong, \emph{Stacks project}. 
Open source project on stacks available 
online at http://stacks.math.columbia.edu.


\bibitem[Del98]{dm} P. Deligne, J. Morgan, 
Notes on supersymmetry (following Joseph Bernstein), in 
\emph{Quantum fields and strings: a course for mathematicians}, Vol. 1, 
41-97, Amer. Math. Soc., Providence, RI, 1998.

\bibitem[FL14]{fl} R. Fioresi, M. A. Lledo, {\it The Minkowski
and conformal superspaces}, World Sci. Publishing, 2014.

\bibitem[RLV07]{flv} R. Fioresi, M. A. Lledo, V. S. Varadarajan,
\textit{The Minkowski and conformal superspaces},
J. Math. Phys., 48, 113505, 2007.

\bibitem[GW10]{GW} U. G\"ortz, T. Wedhorn, \emph{Algebraic geometry I}. 
Vieweg-Teubner-Verlag, 2010.

\bibitem[EGA]{EGA} A. Grothendieck, 
\emph{Elements de G\'eom\'etrie alg\'ebrique}. 
Grundlehren der Mathematischen Wissenschaften (in French) 166 (2nd ed.), 
Springer-Verlag, 1971.


\bibitem[KMS93]{KMS} I. Kolar, P. Michor, J. Slovak,  
\emph{Natural operations in differential geometry}. 
Springer-Verlag, Berlin, 1993.

\bibitem[Kos77]{Kon} B. Kostant, 
\emph{Graded Manifolds, graded Lie theory, and prequantization}. 
In Differential geometrical methods in mathematical physics, II, 
Lecture Notes in Math. 570, Springer-Verlag, Berlin, 1993.

\bibitem[Lei80]{leites} D. A. Leites, 
\emph{Introduction to the theory of supermanifolds}. 
Uspekhi Mat. Nauk, 35(1(211)): 3-57, 255, 1980.

\bibitem[Man88]{ma} Y. I. Manin, 
\emph{Gauge field theory and complex geometry}. 
Translated by N. Koblitz and J.R.King, Grundleheren Math. Wiss. 289, 
Springer-Verlag, Berlin, 1988. 


\bibitem[SS74]{SS} A. Salam, J. Strathdee, 
\emph{Super-gauge transformations}. Nucl. Phys. B, 76, (1974), 477-482. 
vii, 140.

\bibitem[Sch84]{Sch} A. S. Schwarz, \emph{On the definition of superspace}. 
Teoret. Mat. Fiz., 60(1):37???42, (1984).

\bibitem[Shu99]{Shu} V. Shurigin, 
\emph{The structure of smooth mappings over Weil algebras and the 
category of manifolds over algebras}. 
Lobachevskii J. Math. 5: 29-55 (electronic), (1999).



\bibitem[Var04]{vsv} V. S. Varadarajan, \emph{Supersymmetry for 
mathematicians: an introduction}. Courant Lect. 
Notes Math. 11, New York university, 
AMS, Providence, RI, 2004. 


\bibitem[Vis05]{Vis} A. Vistoli, \emph{Grothendieck topologies, 
fibered categories and descent theory}, pg. 1-104, 
in B. Fantechi et al. editors 
\emph{Fundamental algebraic geometry}, Math. surveys and monographs 123, 
AMS, 2005.

\bibitem[WZ74]{WZ} J. Wess, B. Zumino, 
\emph{Supergauge transformations in four dimensions}. 
Nucl. Phys. B, 70 (1974), 39-50, vii, 140.



\end{thebibliography}
\end{document}